\documentclass[a4paper, 11pt]{amsart}

\usepackage[mathscr]{eucal}

\usepackage{amsmath,amssymb,amsfonts,amsthm}

\newtheorem{theorem}{Theorem}[section]
\newtheorem{conjecture}[theorem]{Conjecture}
\newtheorem{lemma}[theorem]{Lemma}

\newtheorem{proposition}[theorem]{Proposition}

\newtheorem*{CDT}{Cauchy-Davenport Theorem}

\newcommand{\N}{\mathbb N}
\newcommand{\Z}{\mathbb Z}

\newcommand{\HH}{\mathsf H}

\newcommand{\Fc}{\mathcal F}
\newcommand{\vp}{\mathsf v}
\newcommand{\D}{\mathsf D}
\newcommand{\h}{\mathsf h}

\newcommand{\ord}{\text{\rm ord}}

\newcommand{\supp}{\text{\rm supp}}

\newcommand{\be}{\begin{equation}}
\newcommand{\ee}{\end{equation}}

\newcommand{\bnml}{\begin{multline}}
\newcommand{\enml}{\end{multline}}

\newcommand{\ber}{\begin{eqnarray}}
\newcommand{\eer}{\end{eqnarray}}
\newcommand{\nn}{\nonumber}
\newcommand{\Sum}[2]{\underset{#1}{\overset{#2}{\sum}}}
\newcommand{\Summ}[1]{\underset{#1}{\sum}}

\newcommand{\und}{\;\;\;\mbox{ and }\;\;\;}

\author{David J. Grynkiewicz}

\address{Institut f\"ur Mathematik und Wissenschaftliches Rechnen \\
Karl-Franzens-Universit\"at Graz \\
Heinrichstra\ss e 36\\
8010 Graz, Austria} \email{diambri@hotmail.com}


\begin{document}
\title{Note on a Conjecture of Graham}
\thanks{Supported by FWF project number M1014-N13}

\begin{abstract}
An old conjecture of Graham stated that if $n$ is a prime and $S$ is a sequence of $n$ terms from the cyclic group $C_n$ such that all (nontrivial) zero-sum subsequences have the same length, then $S$ must contain at most two distinct terms. In 1976, Erd\H{o}s and Szemeredi gave a proof of the conjecture for sufficiently large primes $n$. However, the proof was complicated enough that the details for small primes were never worked out. Both in the paper of Erd\H{o}s and Szemeredi and in a later survey by Erd\H{o}s and Graham, the complexity of the proof was lamented. Recently, a new proof, valid even for non-primes $n$, was given by Gao, Hamidoune and Wang, using Savchev and Chen's recently proved structure theorem for zero-sum free sequences of long length in $C_n$. However, as this is a fairly involved result, they did not believe it to be the simple proof sought by Erd\H{o}s, Graham and Szemeredi. In this paper, we give a short proof of the original conjecture that uses only the Cauchy-Davenport Theorem and pigeonhole principle, thus perhaps qualifying as a simple proof. Replacing the use of the Cauchy-Davenport Theorem with the Devos-Goddyn-Mohar Theorem, we obtain an alternate proof, albeit not as simple, of the non-prime case. Additionally, our method yields an exhaustive list detailing the precise structure of $S$ and works for an arbitrary finite abelian group, though the only non-cyclic group for which this is nontrivial is $C_2\oplus C_{2m}$.
\end{abstract}

\maketitle

\section{Introduction}

The following was an old conjecture of Graham \cite{Erdos-Szermeredi-Graham-conj}.

\bigskip

\begin{conjecture} Let $C_p$ be the cyclic group of order $p$ prime and let $S$ a sequence over $C_p$ of length $p$. If all (nontrivial) zero-sum subsequences of $S$ are of the same length, then the number of distinct terms in $S$ is at most $2$.
\end{conjecture}

\bigskip

In 1976, Erd\H{o}s and Szemeredi gave a proof of the conjecture for sufficiently large primes $p$ \cite{Erdos-Szermeredi-Graham-conj}. However, the proof was complicated enough that the details for small primes were never worked out. Both in the paper of Erd\H{o}s and Szemeredi and in a later survey by Erd\H{o}s and Graham \cite{Erdos-Graham-survey}, the complexity of the proof was lamented. Recently, a new proof, valid even for non-primes, was given by Gao, Hamidoune and Wang \cite{Gao-Ham-Wang-graham-conj}, using Savchev and Chen's recently proved structure theorem for zero-sum free sequences of long length in the cyclic group $C_n$ \cite{sav-chen}. However, as Savchev and Chen's result is fairly involved, they did not believe it to be the simple proof sought by Erd\H{o}s, Graham and Szemeredi.

In this paper, we give a short proof to the original conjecture of Graham that uses only  the Cauchy-Davenport Theorem and pigeonhole principle \cite{natbook} \cite{taobook}. Since the proof of the Cauchy-Davenport Theorem (known since 1813 \cite{cauchy}) is elementary and requires only a paragraph, our proof may perhaps qualify as simple. Replacing the use of the Cauchy-Davenport Theorem with the Devos-Goddyn-Mohar Theorem \cite{Devos-Goddyn-Mohar-thm} (alternatively, the partition theorem from \cite{ccd} \cite{hamconj-paper} could be used instead of Devos-Goddyn-Mohar), we obtain an alternate proof, albeit not as simple, of the non-prime case. With only a little added effort, our method naturally yields an exhaustive list detailing the precise structure of $S$ and shows that the result holds in an arbitrary finite abelian group, though the only additional group for which this is nontrivial (in view of basic bounds on the Davenport constant) is $C_2\oplus C_{2m}$. We state the main theorem in Section 3, after introducing modern notation for sumsets, sequences and subsequence sums.

\bigskip
\section{Notation and Preliminaries} \label{1}
\bigskip

We follow the notation of \cite{GaoGer-survey}, \cite{Alfred-book}, \cite{Alfred-BCN-notes} and \cite{oscar-weighted-projectI} concerning sumsets, sequences and subsequence sums. For the convenience of the reader less familiar with this notation, we give self-contained definitions for all relevant concepts in this section.

\subsection{Sumsets}
Let $G$  be an abelian group, and let  $A,\, B \subseteq G$ be nonempty subsets. Then  $$A+B = \{a+b
\mid a \in A,\, b \in B \}$$  denotes their  \emph{sumset}. For $g\in G$, we let $g+A=\{g+a\mid a\in A\}$ and let $r_{A,B}(g)$ denote the number of representations of $g=a+b$ as a sum with $a\in A$ and $b\in B$. The  \emph{stabilizer}  of $A$ is  $$\mathsf{H}(A): = \{ g \in G \mid g +A = A\}.$$ The order of an element $g\in G$ is denoted $\ord(g)$, and we use $\phi_H:G\rightarrow G/H$ to denote the natural homomorphism modulo $H$. We use $C_n$ to denote the cyclic group of order $n$.

\subsection{Sequences}
We let $\Fc(G)$ denote the free abelian monoid with basis $G$ written multiplicatively. The elements of $\Fc(G)$ are then just multi-sets over $G$, but following long standing tradition, we refer to the $S\in \Fc(G)$ as \emph{sequences}. We write
sequences $S \in \mathcal F (G)$ in the form
$$
S =  s_1\cdots s_r=\prod_{g \in G} g^{\vp_g (S)}\,, \quad \text{where} \quad
\vp_g (S)\geq 0\mbox{ and } s_i\in G.
$$
We call $|S|:=r=\Summ{g\in G}\vp_g(S)$ the \emph{length} of $S$, and  $\mathsf v_g (S)\in \N_0$  the \ {\it multiplicity} \ of $g$ in
$S$. The \emph{support} of $S$ is $$\supp(S):=\{g\in G\mid \vp_g(S)>0\}.$$  A sequence $S_1 $ is called a \emph{subsequence}  of
$S$ if $S_1 | S$  in $\mathcal F (G)$  (equivalently,
$\vp_g (S_1) \leq\vp_g (S)$  for all $g \in G$), and in such case, $S{S_1}^{-1}$ or ${S_1}^{-1}S$ denotes the subsequence of $S$ obtained by removing all terms from $S_1$.
We let $$\h(S):=\max \{\vp_g (S) \mid g \in G \}$$ denote the maximum multiplicity of a term of $S$.
Given any map $\varphi: G\rightarrow G'$, we extend $\varphi$ to a map of sequences,
$\varphi: \Fc(G)\rightarrow \Fc(G')$,  by letting $\varphi(S):=\varphi(s_1)\cdots \varphi(s_r)$.

\subsection{Subsequence Sums} If  $S=s_1\cdots s_r\in \Fc(G)$, with $s_i\in G$, then the \emph{sum} of $S$ is $$\sigma(S):=\Sum{i=1}{r}s_i=\Summ{g\in G}\vp_g(S)g.$$ We say $S$ is \emph{zero-sum} if $\sigma(S)=0$. We adapt the convention that the sum of the trivial/empty sequence is zero.
We follow the usual notation for the set of subsequence sums:
\ber\Sigma_n(S) &=& \left\{\sigma(T)\mid\;\; T|S \mbox{ and } |T|=n\right\}\nn\\\nn
\Sigma_{\leq
n}(S)=\bigcup_{i=1}^{n}\Sigma_i(S)&\und&\Sigma_{\geq n}(S)=\bigcup_{i=n}^{|S|}\Sigma_i(S)\und \Sigma(S)=\Sigma_{\leq |S|}(S).\eer

\subsection{Preliminary Results}
For a finite abelian group $G$, we define the Davenport constant $\D(G)$ to be the minimal integer such that any $S\in \Fc(G)$ with $|S|\geq \D(G)$ has $0\in \Sigma(S)$. A basic argument shows $\D(G)\leq |G|$ (see \cite[Propositions 5.1.4]{Alfred-book}).

We need the following result (see  \cite[Theorem 5.2.10; Lemma 5.2.9]{Alfred-book} and also \cite[Lemma 2.1]{natbook}). Proposition \ref{mult_result}(ii) is a simple consequence of the pigeonhole principle, and we will only use Proposition \ref{mult_result}(i) in the trivial case $|B|=k=2$.

\begin{proposition}\label{mult_result} Let $G$ be an abelian group with $A,\,B\subseteq G$
 finite and nonempty:

(i) if $|A+B|\leq |A|+|B|-k$, then $r_{A,B}(x)\geq k$ for all $x\in A+B$;

(ii) if $G$ is finite and $|A|+|B|\geq |G|+k$, then $r_{A,B}(x)\geq k$ for all $x\in G$.
\end{proposition}

Next, we state a special case of the Devos-Goddyn-Mohar Theorem \cite{Devos-Goddyn-Mohar-thm}.

\begin{theorem}\label{DGM-devos-etal-thm} Let $G$ be an abelian group, let $S\in \Fc(G)$ be a sequence, and let $n\in \Z^+$ with $n\leq |S|$. If $H=\HH(\Sigma_n(S))$, then \be\label{DGM-bound}|\Sigma_n(S)|\geq (\Summ{g\in G/H}\min\{n,\,\vp_g(\phi_H(S))\}-n+1)|H|.\ee
\end{theorem}

A particular case of the (general) Devos-Goddyn-Mohar Theorem is the much simpler Cauchy-Davenport Theorem \cite{cauchy} \cite{cdt} \cite{natbook} \cite{taobook}.

\begin{CDT} Let $p$ be prime and let $A_i\subseteq C_p$, for $i=1,\ldots,n$, be nonempty. Then $$|\Sum{i=1}{n}A_i|\geq \min\{\Sum{i=1}{n}|A_i|-n+1,\,p\}.$$
\end{CDT}

\bigskip
\section{When the Length of a Zero-Sum is Unique}
\bigskip

We begin with the following simple lemma.

\begin{lemma}\label{neededlemma} Let $G$ be an abelian group, let $g\in G$, and let $R\in \Fc(G)$ be nontrivial with \be\label{gg1} \Sigma(R)\subseteq \{g,2g,\ldots,|R|g\}.\ee If $|R|\leq \ord(g)-1$ and $\sigma(R)=|R|g$, then $R=g^{|R|}$.
\end{lemma}

\begin{proof} The result is clear when $|R|\leq 2$, so we may assume $|R|\geq 3$.
In view of \eqref{gg1} and $|R|\leq \ord(g)-1$, we have $0\notin \Sigma(R)$. Suppose to the contrary that there is \be\label{iit}h\in \supp(R)\subseteq \Sigma(R)\subseteq \{g,2g,\ldots,|R|g\}\ee with $h\neq g$. Note, since $|R|\leq \ord(g)-1$, that \eqref{iit} shows $h\neq 0$ as well. From $0\notin \Sigma(R)$ and \eqref{gg1} (note if $R'|Rh^{-1}$ with $\sigma(R')=\sigma(R)$, then $\sigma(R{R'}^{-1})=0$), we have $$\Sigma(Rh^{-1})\subseteq (\{g,2g,\ldots, |R|g\}\setminus \{\sigma(R)\})\cap (\{g,2g,\ldots,|R|g\}-h).$$ Consequently, $h\notin \{g,0\}$, $0\notin \{g,2g,\ldots, |R|g\}$ and $\sigma(R)=|R|g$ imply that $|\Sigma(Rh^{-1})|<|Rh^{-1}|=|R|-1$.
As a result, $|\Sigma(Rh^{-1})\cup \{0\}|=|\Sum{i=1}{|R|-1}\{0,g_i\}|\leq |R|-1$, where $Rh^{-1}=g_1\cdots g_{|R|-1}$ with $g_i\in \supp(R)\subseteq \Sigma(R)\subseteq G\setminus \{0\}$. Hence Proposition \ref{mult_result}(i) (applied to the partial sums $\Sum{i=1}{j-1}\{0,g_i\}+\{0,g_j\}$) implies every element of $\Sum{i=1}{|R|-1}\{0,g_i\}$ has at least two representations, contradicting that $0\notin \Sigma(R)$.
\end{proof}

The next two lemmas will help with the detailed characterization of $S$.

\begin{lemma}\label{double-gen-lemm} Let $g,\,h\in C_n$ and let $S\in \Fc(C_n)$ with $S=g^lh^{n-l}$ and $l\geq n-l\geq 1$. Suppose $g$ is a generator and \be\label{sfut}\Sigma(h^{n-l})=\{g,2g,\ldots,(n-l-1)g\}\cup\{b_0\},\ee for some $b_0\in C_n$. If there is a unique $r\in [1,n]$ such that $0\in \Sigma_r(S)$, then either  $S=g^{n-1}h$ or else $n$ is odd, $h=\frac{n+1}{2}g$ and $S=g^{n-2}h^2$.
\end{lemma}

\begin{proof}  The cases $n-l\leq 2$ and $l\leq 1$ are easily verified, so we may assume $3\leq n-l\leq n-2$ and thus $h\neq \pm g$ (else either there are two disjoint zero-sums of length $2$ or $S=g^{n-1}h=g^n$). Now \eqref{sfut} implies $$\Sigma(h^{n-l})=\{h,2h,\ldots,(n-l)h\}= \{g,2g,\ldots,(n-l-1)g\}\cup\{b_0\},$$ for some $b_0\in C_n$. Thus $\{h,2h,\ldots,(n-l)h\}$ contains an arithmetic progression of difference $g\neq \pm h$ and length $n-l-1\geq 2$.  Consequently, $h$ must also be a generator. Hence, if $n-l\geq 4$, then it is easily seen, in view of the hypothesis $n-l\leq \frac{n}{2}$, that $\{h,2h,\ldots,(n-l)h\}$ cannot contain an arithmetic progression of length $n-l-1$ and difference $g\neq \pm h$. On the other hand, if $n-l=3$, then this could only be possible if $g=\pm 2h$, and this final case can be eliminated by individual consideration, completing the proof.
\end{proof}

\begin{lemma}\label{non-gen-lemm} Let $G$ be an abelian group of order $n$ even, let $g,\,h\in G$ with $\ord(g)=\frac{n}{2}$ and $h\neq g$, and let $S\in \Fc(G)$ with $S=g^l{h}^{n-l}$, $n-l\geq 2$ and $l\geq \frac{n}{2}$. If $\frac{n}{2}\in [1,n]$ is the unique integer $r$ such that $0\in \Sigma_{r}(S)$, then $n-l$ is odd, $h\notin \langle g\rangle$ and $2h=2g$.
\end{lemma}

\begin{proof} Since $h\neq g$, $l\geq \frac{n}{2}$ and $\ord(g)=\frac{n}{2}\in [1,n]$ is the unique integer $r$ such that $0\in \Sigma_{r}(S)$, we conclude that $h\notin \langle g\rangle$. However, noting that $2h\in \langle g\rangle$ (since $\langle g\rangle$ has index $2$), we likewise see that we must have $2h=2g$ (in view of $n-l\geq 2$), else the uniqueness of $\frac{n}{2}=\ord(g)$ is again contradicted. Consequently, the sum of any $\frac{n}{2}$-terms of $S$ using an even number of terms from $h^{n-l}$ has sum zero. As a result, if $n-l$ is even, then there are two disjoint zero-sum subsequences of length $\frac{n}{2}$, contradicting the uniqueness of $\frac{n}{2}$, and completing the proof.
\end{proof}

Next, we state and prove the main result. In the remark that follows the proof of Theorem \ref{thm-gen-graham}, we explain how the proof can be simplified in the case $G=C_p$ with $p$ prime, including the use of the Cauchy-Davenport Theorem in place of Devos-Goddyn-Mohar. Also, though we state the theorem for an arbitrary finite abelian group, most non-cyclic cases have no sequences satisfying the hypotheses (since $2\D(G)\leq |G|$ holds for most non-cyclic groups \cite[Theorem 5.5.5]{Alfred-book}.) The proof is divided into four main sections labeled steps.

\begin{theorem}\label{thm-gen-graham}
Let $G$ be a finite abelian group of order $n$ and let $S\in \Fc(G)$ with $|S|=n$. Suppose there is a unique $r\in [1,n]$ such that $0\in \Sigma_r(S)$. Then $|\supp(S)|\leq 2$.

If $G$ is non-cyclic, then $G=\langle h\rangle \oplus \langle g\rangle\cong C_2\oplus C_{2m}$, $r=\frac{n}{2}=2m$ and
$$S=g^{n-1}g'\quad\mbox{ or }\quad S=g^{n/2+x}(h+g)^{n/2-x}\quad\mbox{ or }\quad S=g^{n/2+x}(h+\frac{n+4}{4}g)^{n/2-x},$$
where $g\in G$, $h,\,g'\in G\setminus \langle g\rangle$, $\ord(g)=\frac{n}{2}$, $\ord(h)=2$ and  $x\in [1,\frac{n}{2}-1]$ is odd.

If $G$ is cyclic, then there exists a generator $g\in G\cong C_n$ such that either $$S=g^{n-1}g'\quad\mbox{ or }\quad S=(2g)^{n-1}g'',$$ for some $g'\in G$ or $g''\in G\setminus \langle 2g\rangle$; or $n$ is odd, $r=\frac{n+1}{2}$ and $$S=g^{n-2}(\frac{n+1}{2}g)^2;$$
or $n\equiv 2\mod 4$, $r=\frac{n}{2}$ and $$S=(2g)^{n/2+x}(\frac{n+4}{2}g)^{n/2-x},$$
where $x\in [0,\frac{n}{2}-1]$ is even; or $n$ is even, $r=\frac{n}{2}$ and $$S=g^{n/2+x}(\frac{n+2}{2}g)^{n/2-x},$$ where $x\in [0,\frac{n}{2}-1]$ with $\frac{n}{2}-x$ odd.
\end{theorem}

\begin{proof}  Recalling the well-known fact that a zero-sum free subsequence of length $|G|-1$ must be of the form $g^{|G|-1}$ for a generator $g\in G$ (this can be proved in a few lines using the trivial case $|B|=k=2$ in Proposition \ref{mult_result}(i); see also \cite[Lemma 5.4.2]{Alfred-book} for a slightly more involved proof), we see that the cases $r=1$ and $r=n$ are trivial. Therefore we assume $1<r<n$, whence $0\notin \supp(S)$.
Observe that \ber\label{bee1}0&\notin& \Sigma_{\leq r-1}(S)=\Sigma_{r-1}(0^{r-2}S),\\ 0&\notin& \Sigma_{\geq r+1}(S)=\Sigma_n(0^{n-r-1}S)=\sigma(S)-\Sigma_{n-r-1}(0^{n-r-1}S),\label{bee2}\eer where we have used for \eqref{bee2} the fact that $\Sigma_m(T)=\sigma(T)-\Sigma_{|T|-m}(T)$ for $T\in \Fc(G)$, which follows in view of the correspondence between $R|T$ and $TR^{-1}|T$.

\subsection*{Step 1.} Let $g\in \supp(S)$ be a term with $\vp_g(S)=l:=\h(S)$. We first show that either \ber\label{mult-hyp}\h(S)&\geq& \max\{r,\, n-r+1\}, \mbox{ or}\\ \label{mult-hyp-alt}\h(S)&\geq& \max\{r,\, n-r\}\und S=g^{n/2}{g'}^{n/2}\mbox{ with } \ord(g)=\ord(g')=n\mbox{ even},\eer where $g'\in G$.
We do so in two cases. First suppose \be\label{assump1}n-r-1\geq \frac{n}{2}-1,\ee in which case $n-r+1> r$. Note that if there are distinct $g,\,g'\in \supp(S)$ each with multiplicity at least $n-r$, then  \eqref{assump1} implies $n$ is even with $S=g^{n/2}{g'}^{n/2}$  and $r=\frac{n}{2}$.  If $\ord(g)=\ord(g')=n$, then \eqref{mult-hyp-alt} holds, as desired. On the other hand, if (say) $\ord(g)\leq \frac{n}{2}$, then $\ord(g)=r=\frac{n}{2}$, in which case the proof is easily concluded using Lemma \ref{non-gen-lemm}. Therefore we may assume there is at most one term with multiplicity at least $n-r$.

We apply Theorem \ref{DGM-devos-etal-thm} to $\Sigma_{n-r-1}(0^{n-r-1}S)$. Let $H=\HH(\Sigma_{n-r-1}(0^{n-r-1}S))$. Now assuming $\h(S)\leq  n-r$, it follows, in view of  \eqref{DGM-bound} and \eqref{bee2}, and since there is at most one term of multiplicity $n-r$, that $H$ is a proper, nontrivial subgroup. Moreover, in view of $$\vp_0(0^{n-r-1}S)=n-r-1\geq \frac{n}{2}-1\geq |G/H|-1,$$ which follows from \eqref{assump1},
we see that \eqref{DGM-bound} implies that all but at most $|G/H|-2$ terms of $S$ are from $H$. Letting $T|S$ be the subsequence of all terms not from $H$, we see that $\sigma(S)\in \sigma(T)+H$. Thus, since $|T|\leq |G/H|-2\leq \frac{n}{2}-1\leq n-r-1$ (by \eqref{assump1}), it follows, in view of the definition of $H$, that $\sigma(S)\in \Sigma_{n-r-1}(0^{n-r-1}S)$, in contradiction to \eqref{bee2}. Therefore we may instead assume \eqref{assump1} fails, i.e, \be\label{assump2} r-1> \frac{n}{2}-1.\ee

In this case, we apply the Theorem \ref{DGM-devos-etal-thm} to $\Sigma_{r-1}(0^{r-2}S)$. However, assuming $\h(S)\leq r-1$ and repeating the above arguments using \eqref{bee1} instead of \eqref{bee2} and using \eqref{assump2} instead of \eqref{assump1}, we arrive at the same contradiction. Therefore we conclude that $\h(S)\geq r> \frac{n}{2}>n-r$, as claimed. Thus \eqref{mult-hyp} is established in both cases.

\bigskip

Factor $S=g^lT$, where $T\in \Fc(G)$, and let $R|T$ be a maximal length subsequence (possibly trivial) such that $\sigma(R)=|R|g$.
In view of \eqref{mult-hyp}, \eqref{mult-hyp-alt} and \eqref{bee1}, it follows that \be\label{l-bigg}\vp_g(S)=l=\h(S)\geq \max\{r,\,n-r\}\geq \frac{n}{2}\geq |T|\ee and $0\notin \Sigma(T)$; in particular, $0\notin \Sigma(R)$.

\subsection*{Step 2.} Suppose $\ord(g)<n$. Then it follows in view of \eqref{l-bigg} that $r=\ord(g)$ and that $g$ is the only element from $H:=\langle g\rangle$ in $\supp(S)$ (else we can find a zero-sum of length distinct from $r$). Iteratively applying the definition of $\D(G/H)\leq |G/H|$ to $\phi_H(U^{-1}Sg^{-\ord(g)})$, beginning with $U$ trivial, we find a zero-sum mod $H$ subsequence $U|Sg^{-\ord(g)}$ with $|U|\geq n-|H|-|G/H|+1$. Adding on an appropriate number of terms from $g^{\ord(g)}$ (note $\Sigma(g^{\ord(g)})=H$) yields a zero-sum subsequence $U'|S$ with $|U'|\geq n-|H|-|G/H|+2$. If $|H|<\frac{n}{2}$, then $|U'|>|H|=r$, a contradiction. On the otherhand, if $|H|=\frac{n}{2}$, then we obtain the same contradiction unless $|U|=\frac{n}{2}-1$, $\sigma(U)=-g$ and $SU^{-1}g^{-n/2}=g_0\notin H$. Thus, if there is some $g'_0\in \supp(T)\setminus H$ with $g'_0\neq g_0$, then swapping $g_0$ for $g'_0$ in $U$ yields a new $U|Sg^{-\ord(g)}$ with $\sigma(U)\in H$ and $|U|=\frac{n}{2}-1$ but $\sigma(U)\neq -g$, whence we obtain the contradiction as before. Therefore, we instead see that all terms outside $H$ in $\supp(S)$ are equal to $g_0$. However, since all terms inside $H$ in $\supp(S)$ are equal to $g$, this shows $|\supp(S)|\leq 2$. But now the proof is easily concluded using Lemma \ref{non-gen-lemm}. So we henceforth assume $\ord(g)=n$, in which case $G\cong C_n$ is cyclic.

\bigskip

Since
$$|R|\leq |T|\leq r\leq l\leq n-2= \ord(g)-2$$ (the last inequality holds else the proof is complete, while the other inequalities follow from \eqref{l-bigg}), and since $\sigma(R)=|R|g$, it follows that \ber\label{lonestar}0\notin\{g,2g,\ldots,rg\}&\subseteq& \Sigma(g^l),\\\label{lonestar2} 0\notin \{(r+1)g,(r+2)g,\ldots,(l+|R|)g\}&\subseteq& \Sigma_{\geq r+1}(g^lR).\eer Hence $l+|R|\leq \ord(g)-1=n-1$ and $$|R|<|T|=n-l\leq r.$$

\subsection*{Step 3.} Next, we show that, when $R$ is nontrivial, there is some \be\label{h-exists}h\in \Sigma_{\geq r+1}(g^lR)\setminus \{g,2g,\ldots, (l+|R|)g\}.\ee Thus assume for the moment that $R$ is nontrivial. Then, in view of Lemma \ref{neededlemma} and $0\notin \Sigma(R)$, there is some nontrivial $R_0|R$ with  $\sigma(R_0)\notin \{0,g,\ldots,|R|g\}$. Note $\sigma(R_0)\neq |R|g=\sigma(R)$ implies $|R_0|<|R|<|T|\leq r$; thus $1\leq |R_0|\leq r-2$. If $\sigma(R_0)\in \{-g,-2g,\ldots, -(r-|R_0|-1)g\}$, then $0\in \Sigma_{\leq r-1}(g^lR_0)$, contrary to hypothesis.
Therefore $$\sigma(R_0)\in \{(|R|+1)g,(|R|+2)g,\ldots,(n-r+|R_0|)g\},$$ whence $l+|R|\leq \ord(g)-1=n-1$ and $g^l|S$ with $l\geq r\geq r-|R_0|+1\geq 0$ show that
either $\sigma(R_0)=(n-r+|R_0|)g=(|R_0|-r)g$ or else \eqref{h-exists} holds, as desired. However, in the former case, factor $R=R_0R_1$ and note that $\sigma(R_1)=\sigma(R)-\sigma(R_0)=(|R_1|+r)g$. Now $$|R|<|T|<|R_1|+r\leq |T|-1+r\leq n-1,$$ where the last inequality follows from $|T|=n-l$ with $l\geq r$, whence $\sigma(R_1)\notin \{0,g,\ldots,|R|g\}$ (in view of $\ord(g)=n$) and so $|R_1|<|R|<|T|\leq r$. Thus $1\leq |R_1|\leq r-2$, and applying the above arguments with $R_1$ instead of $R_0$, we establish \eqref{h-exists} unless $(|R_1|+r)g=\sigma(R_1)=(|R_1|-r)g$. However, the latter case implies $2rg=0$, whence $r=\frac{n}{2}$ with $n$ even.

Furthermore, by the above work for $R_0$ and $R_1$, we see that \eqref{h-exists} is established unless \be\label{trisket}\sigma(R')\in \{g,2g,\ldots,|R|g\}\cup \{(|R'|-\frac{n}{2})g\}\ee for all nontrivial $R'|R$. Applying \eqref{trisket} to each $x\in \supp(R)$,
noting that $g\notin \supp(R)$ (in view of $R|T$), and recalling that $|R|<|T|\leq r=\frac{n}{2}$,
we conclude that \be\label{stael}\nn\supp(R)\subseteq \{2g,3g,\ldots,(\frac{n}{2}-1)g\}\cup \{(\frac{n}{2}+1)g\}.\ee If there are $(\frac{n}{2}+1)g,\,x\in \supp(R)$ with $x\in \{2g,3g,\ldots,(\frac{n}{2}-1)g\}$, then applying \eqref{trisket} to the sequence $x((\frac{n}{2}+1)g)$ yields a contradiction. Therefore we conclude that either \be\label{deel}\supp(R)=\{(\frac{n}{2}+1)g\}\;\;\;\mbox{ or }\;\;\;\ \supp(R)\subseteq \{2g,3g,\ldots,(\frac{n}{2}-1)g\}.\ee

Noting that $\frac{n}{2}g=rg=-rg$ and $\sigma(R_1)=(|R_1|+r)g$, we see that $$\{(|R_1|+r+1)g,(|R_1|+r+2)g,\ldots ,(n-2)g\}\subseteq \sigma(R_1)+\Sigma_{\leq r-|R_1|-2}(g^l)\subseteq \Sigma_{\leq r-2}(g^lR_1).$$ Thus, since $|R_1|+r+1\leq |R|+l+1$, and in view of \eqref{lonestar2} and $\ord(g)=n$, we have \be\label{kush}\nn G\setminus \{-g,(r-1)g\}\subseteq \Sigma_{\leq r-2}(g^lR)\cup \Sigma_{\geq r}(g^lR).\ee As a result (recall $|R|<|T|$),
\ber\label{zat}\supp(TR^{-1})&=& \{-(r-1)g\}=\{(\frac{n}{2}+1)g\},\\\label{zet} (r-1)g&\notin &\Sigma_{\leq r-2}(g^lR)\cup \Sigma_{\geq r}(g^lR),\eer else we find a zero-sum of length distinct from $r$ using precisely one term from $TR^{-1}$ (recall $g\notin \supp(T)$ in view of the definition of $T$), contrary to hypothesis.

By \eqref{deel} and \eqref{zat}, we discover that $\supp(R)\subseteq \{2g,3g,\ldots,(\frac{n}{2}-1)g\}$, else $\supp(S)=\{g,\frac{n+2}{2}g\}$ with $r=\frac{n}{2}$, from which the remainder of the proof is easily deduced. Thus, since $r=\frac{n}{2}$ and $R$ is nontrivial, we see that $r\geq 3$ and  $(tg)g^{r-1-t}|g^lR$ for some $t\in [2,r-1]$. However $\sigma((tg)g^{r-1-t})=(r-1)g$ with $|(tg)g^{r-1-t}|=r-t\in [1,r-2]$, contradicting \eqref{zet}. So we see that \eqref{h-exists} is finally established.

\subsection*{Step 4.}
Let $A:=\{g,2g,\ldots,(l+|R|),h\}$ if $R$ is nontrivial, and otherwise let $A:=\{g,2g,\ldots, lg\}$. Let $TR^{-1}=g_1\cdots g_{n-|R|-l}$, where $g_i\in G$. Recall $|R|<|T|$, so $TR^{-1}$ is nontrivial.
Let $T_i:=g_1\cdots g_i$, for $i=0,1,\ldots,n-{|R|}-l$. Now $B:=\{\sigma(T_0),\sigma(T_1),\sigma(T_2),\ldots,\sigma(T_{n-{|R|}-l})\}$ is a set of cardinality $n-l-{|R|}+1$ by the following reasoning: if  $\sigma(T_i)=\sigma(T_j)$ with $i<j$, then  $\sigma(T_j{T_i}^{-1})=0$, which contradicts $0\notin \Sigma(T)$.
Note that $A+B=G$ in view of Proposition \ref{mult_result}(ii); moreover, if $|R|>0$, then every element has at least two representations.

Suppose $0\in (A+\sigma(T_i))\cap (A+\sigma(T_j))$ for some $i<j$, i.e., $0$ has at least two representations, say $0=x_ig+\sigma(T_i)$ and $0=x_jg+\sigma(T_j)$, as a sum in $A+B$, where $x_i,\,x_j\in [1,n]$. Consequently, since (from \eqref{lonestar2}) $$\{(r+1)g,(r+2)g,\ldots,(l+|R|)\}\subseteq \Sigma_{\geq r+1}(g^lR),$$ and since $h\in \Sigma_{\geq r+1}(g^lR)$ if $R$ is nontrivial (from \eqref{h-exists}), we see from the definition of $A$ that $x_i,\,x_j\in [1,r]$, else $0\in \Sigma_{\geq r+1}(S)$, contrary to hypothesis. Thus $g^{x_i}T_i|S$ and $g^{x_j}T_j|S$ are zero-sum subsequences, and so
our hypothesis of all zero-sums having length $r$ implies $\sigma(T_i)=(|T_i|-r)g$ and $\sigma(T_j)=(|T_j|-r)g$, whence $\sigma(T_jT_i^{-1})=|T_jT_i^{-1}|g$. But now $RT_jT_i^{-1}$ contradicts the maximality of $R$. Therefore we may instead assume $0$ has a unique representation in $A+B$, in which case $R$ is trivial, as remarked in the previous paragraph.

However, in this case $A=\{g,2g,\ldots,lg\}$ is an arithmetic progression with difference $g$ such that $0\in A+B=G$ is a unique expression element. Hence it follows that $$|B\cap \{-lg,-(l-1)g,\ldots,-g\}|=1.$$ Let $b_0\in B\cap \{-l,-(l-1)g,\ldots,-g\}$, so that (in view of $|B|=|G|-|A|+1=n-l+1$) \be\label{bitar}B=\{0,g,2g,\ldots,(n-l-1)g\}\cup \{b_0\}.\ee Observe, in view of \eqref{bitar} and Lemma \ref{double-gen-lemm}, that it now suffices to show $|\supp(S)|\leq 2$ to complete the proof. Let $T_k$ be the subsequence such that $\sigma(T_k)=b_0$.

Note that if we swap the index between $g_i$ and $g_{i+1}$, for $i\in [1,k-1]$, and use this ordering to define a new $B$, let us call it $B'$, as above, then $b_0\in B'$ and only one element of $B'$ differs from $B$, namely that corresponding to $\sigma(T_i)$. However, applying the above argument using $B'$ instead of $B$, we see that we again contradict the maximality of $R$ unless $B=B'$ (in view of $b_0\in B'$). As $B=B'$ if and only if $g_i=g_{i+1}$, we conclude that $g_1=g_2=\ldots=g_k$. Likewise, swapping the index between $g_i$ and $g_{i+1}$, for $i\in [k+1,n-l-1]$, and proceeding as we did for $i\in [1,k-1]$ allows us to conclude $g_{k+1}=g_{k+2}=\ldots=g_{n-l}$. Let $g_1=ag$ and $g_{n-l}=bg$, with $a,\,b\in [2,n-1]$ (since $0,\,g\notin \supp(T)$). If $|T|\geq 3$, then we can find an ordering of the $g_i$ such that $g_1=g_{n-l}$. Then using this ordering to define $B$ and repeating the above arguments, we either contradict the maximality of $R$ or show $|\supp(S)|=2$, in which case the proof is complete as remarked before. So it only remains to consider the case $|T|=2$, as the proof is trivially complete when $|T|=1$. But in this case, $l=n-2$ and $a,\,b\in [2,n-1]$ imply that $g^{n-a}(ag)|S$ and $g^{n-b}(bg)|S$ are both zero-sum subsequences of respective lengths $n-a+1$ and $n-b+1$, whence the uniqueness of $r$ as a zero-sum length implies $a=b$. Thus $\supp(S)=\{g,ag\}$, completing the proof as remarked before.
\end{proof}

\subsection*{Remark} When $G=C_n$ with $n$ prime, the above proof can be simplified. First remark that $\ord(g)=n$ holds trivially for $|G|=p$ prime, so Step 2 is unnecessary. Next, noting that the case $n=2$ is trivial, we can assume $n\geq 3$, and thus that $n$ is odd. This eliminates the lengthy extra portion of  Step 3 needed to establish \eqref{h-exists} when $r=\frac{n}{2}$ with $n$ even.
Also, the following argument, using the Cauchy-Davenport Theorem instead of the Devos-Goddyn-Mohar Theorem, can be used to establish \eqref{mult-hyp}.

To show \eqref{mult-hyp}, we proceed in the same two cases. First suppose \be\label{PPassump1}n-r-1\geq \frac{n}{2}-1,\ee in which case $n-r+1> r$. Note that if there are two distinct $g,\,g'\in \supp(S)$ with multiplicity at least $n-r$, then this contradicts \eqref{PPassump1} in view of $n$ odd, whence we may assume otherwise. Thus, assuming $\h(S)\leq  n-r$, it is easily seen that we can find $n-r-1$ nonempty sets $A_1,\ldots, A_{n-r-1}\subseteq G$ such that $\prod_{i=1}^{n-r-1}\prod_{g\in A_i}g=0^{n-r-1}Sx^{-1}\in \Fc(G)$, for some $x\in \supp(S)$ (see \cite[Proposition 2.1]{EGZ-II}). Applying the Cauchy-Davenport Theorem to $A_1,\ldots,A_{n-r-1}$, we find that $\Sigma_{n-r-1}(0^{n-r-1}Sx^{-1})=G$, whence $\Sigma_{n-r-1}(0^{n-r-1}S)=G$, contradicting \eqref{bee2}. Therefore we may instead assume \eqref{PPassump1} fails, i.e, \be\label{PPassump2} r> \frac{n}{2}.\ee
In this case, assuming $\h(S)\leq r-1$, we can (as before) find  $r-1$ nonempty sets $A_1,\ldots, A_{r-1}\subseteq G$ such that $\prod_{i=1}^{r-1}\prod_{g\in A_i}g=0^{r-2}S\in \Fc(G)$. Applying the Cauchy-Davenport Theorem to $A_1,\ldots,A_{r-1}$, we find that $\Sigma_{r-1}(0^{r-2}S)=G$, contradicting \eqref{bee1}. Therefore we conclude, in view of \eqref{PPassump2}, that $\h(S)\geq r> \frac{n}{2}>n-r$, as claimed. Thus \eqref{mult-hyp} is established in both cases.

\end{document}